\documentclass[twocolumn,amsmath,amssymb,aps, physrev,onecolumn,preprint]{revtex4-1}

\usepackage{graphicx}
\usepackage{amsmath}
\usepackage{amssymb}
\usepackage{algorithm}
\usepackage[update,prepend]{epstopdf}
\usepackage{subfigure}
\usepackage{amsthm}
\usepackage{enumerate}
\usepackage{color}

\def\fatx{\mathbf{x}}

\def\fatnu{\mathbf{\nu}}

\def\fatmu{\boldsymbol{\mu}}

\def\fatx{\mathbf{x}}

\def\fatnu{\boldsymbol{\nu}}

\def\kme{k_{\rm CK}}

\def\krme{k_r^{\rm meso}}
\def\krmi{k_r}
\def\kdme{k_d^{\rm meso}}
\def\kdmi{k_d}

\def\hhlratetwo{\rho^{(2)}}
\def\hhlratethree{\rho^{(3)}}
\def\hhlrated{\rho^{(d)}}
\def\hhlratetwofull{\rho^{(2)}(\krmi,h)}
\def\hhlratethreefull{\rho^{(3)}(\krmi,h)}
\def\hhlratedfull{\rho^{(d)}(\krmi,h)}

\def\micromrttwo{\tau^{(2)}_{\rm micro}}
\def\micromrtd{\tau^{(d)}_{\rm micro}}
\def\micromrtthree{\tau^{(3)}_{\rm micro}}

\def\mesomrtd{\tau^{(d)}_{\rm meso}}

\def\hstarirr{h^{\ast}_{\krmi}}
\def\hstarrev{h^{\ast}_{\krmi,\kdmi}}
\def\hstarinf{h^{\ast}_{\infty}}

\def\rradius{\sigma}
\def\diffconst{D}

\def\voxelsize{h}

\def\Gdfull{G^{(d)}(\voxelsize,\rradius)}
\def\Gd{G^{(d)}}

\def\Gthreefull{G^{(3)}(\voxelsize,\rradius)}

\def\Calphatwo{C_{2}}
\def\Calphathree{C_{3}}
\def\Calphad{C_{d}}

\def\numvox{N}
\def\voxsize{h}

\def\smolpdffull{p(\bold{r},t|\bold{r}_n,t_n)}
\def\smolpdf{p}
\def\smolinit{p(\bold{r},t_n|\bold{r}_n,t_n) = \delta(\bold{r}-\bold{r}_n)}
\def\smolboundaryshort{K\frac{\partial \smolpdf}{\partial n}\bigg|_{|\bold{r}|=\rradius} = \krmi \smolpdffull}

\def\Kdef{K = \begin{cases}
4\pi\sigma^2 \diffconst\quad (3D)\\
2\pi\sigma \diffconst\quad (2D).
\end{cases}}

\newtheorem{theorem}{Theorem}

\newtheorem{corollary}{Corollary}

\begin{document}

\begin{abstract}
The mesoscopic reaction-diffusion master equation (RDME) is a popular modeling framework, frequently applied to stochastic reaction-diffusion kinetics in systems biology. The RDME is derived from assumptions about the underlying physical properties of the system, and it may produce unphysical results for models where those assumptions fail. In that case, other more comprehensive models are better suited, such as hard-sphere Brownian dynamics (BD). Although the RDME is a model in its own right, and not inferred from any specific microscale model, it proves useful to attempt to approximate a microscale model by a specific choice of mesoscopic reaction rates. In this paper we derive mesoscopic scale-dependent reaction rates by matching certain statistics of the RDME solution to statistics of the solution of a widely used microscopic BD model: the Smoluchowski model with a Robin boundary condition at the reaction radius of two molecules. We also establish fundamental limits on the range of mesh resolutions for which this approach yields accurate results, and show both theoretically and in numerical examples that as we approach the lower fundamental limit, the  mesoscopic dynamics approach the microscopic dynamics. We show that for mesh sizes below the fundamental lower limit, results are less accurate. Thus, the lower limit determines the mesh size for which we obtain the most accurate results.
\end{abstract}

\title{Reaction rates for mesoscopic reaction-diffusion kinetics}
\author{Stefan Hellander}
\affiliation{Department of Computer Science, \mbox{University of California, Santa Barbara}, CA 93106-5070 Santa Barbara, USA.}
\author{Andreas Hellander}
\affiliation{\mbox{Department of Information Technology, Uppsala University}, Box 337, SE-75105, Uppsala, Sweden.}
\author{Linda Petzold}
\affiliation{Department of Computer Science, \mbox{University of California,~Santa Barbara}, CA 93106-5070 Santa Barbara, USA.}

\maketitle

\section{Introduction}

The reaction-diffusion master equation (RDME) is a commonly used mesoscopic model in the field of computational systems biology and it is a natural extension of the classical well-mixed Markov-process formalism for reaction kinetics \cite{gillespie}. Having a long history in the study of fluctuations in chemical reaction systems \cite{gardiner1976, Baras}, recently it has been successfully applied to study diverse biological phenomena such as yeast polarization \cite{Lawson:2013}, pattern formation in \emph{E. Coli} \cite{Howard,FaEl}, and noisy oscillations of Hes1 in embryonic stem cells \cite{Sturrock1:2013, Sturrock2:2013}. 

In the RDME framework, spatial heterogeneity is modeled by dividing space into voxels in a computational mesh, where molecules are assumed to be well mixed 
 inside each voxel. In individual voxels, reactions are simulated using the Stochastic Simulation Algorithm (SSA) \cite{gillespie}, while diffusion is accounted for through discrete jumps of molecules between voxels. Discrete diffusion and well-mixed SSA are combined in the Next Subvolume Method (NSM) \cite{ElEh04}, an efficient kinetic Monte Carlo method. For moderate mesh resolutions, simulations of the RDME are typically orders of magnitude faster than microscopic particle-tracking models, such as the popular Green's Function Reaction Dynamics (GFRD) algorithm \cite{ZoWo5a,ZoWo5b,SHeLo11}. This contributes to the popularity of the method for applications where the system of interest needs to be studied on the minute to hour timescales that are typical for cellular events like gene expression, signaling and cell division.  The RDME underlies software for spatial stochastic simulation such as MesoRD \cite{mesord}, URDME \cite{URDME_BMC}, pyURDME (www.pyurdme.org), and STEPS \cite{steps}. 

Despite the proven usefulness of the RDME---and the extensive work put into speeding up simulations using approximate \cite{RosBayKou,marquez-lago:104101,LaGiPe} and hybrid  \cite{FeHeLo2009} methods, as well as extending it to include additional transport phenomena \cite{HeLo:2010, Bayati:2013} and to simulate it on complex geometries \cite{IsP,EnFeHeLo}---its fundamental numerical properties and its ability to approximate microscopic particle tracking models at high mesh resolution remains poorly understood. In order to discuss the accuracy of the RDME on small length- and timescales, we need to specify a fine-scaled alternative as our gold standard. A microscale model often utilized for that purpose is the Smoluchowski model, in which particles are modeled by hard spheres that diffuse according to Brownian motion, and reactions are modeled by a partially absorbing boundary condition at the surface of the spheres. This model has a long history in chemical physics, going back to ideas by Smoluchowski \cite{Smol}. In systems biology, the Smoluchowski model is being popularized through software packages such as E-Cell \cite{Tomita01011999} and Smoldyn \cite{AnAdBrAr10}. This paper is concerned with the accuracy of the RDME when viewed as an approximation to the Smoluchowski model.

The principal way in which the mesoscale and the microscale are connected is through the mesoscopic bimolecular reaction constant. A classical result by Collins and Kimball  \cite{Kimball} provides effective rates in terms of the microscopic, intrinsic, reaction parameters of the Smoluchowski model. When scaled by the volume of the voxels, they can be used as approximate mesoscale reaction rates in the RDME for simulations in 3D. The same constant was derived more recently from first principle physics by Gillespie \cite{gillespie:164109}. The constant is valid when voxels are large in comparison to the molecules. Since, in the conventional implementation, reactions occur only between molecules occupying the same voxel, the average time until molecules react diverges with vanishing size of the voxels \cite{Isaacson2, ErCha09,HHP}. A consequence is a lower bound on the size of the voxels, below which no mesoscopic reaction rate can be chosen so that the average reaction times match in the RDME and Smoluchowski models \cite{HHP}.

In previous work \cite{HHP} we analyzed the scenario of a single, irreversible, bimolecular reaction on a Cartesian mesh. For the case of perfect absorption, we obtained analytical lower limits on the mesh size in both 2D and 3D. Above those limits, it is theoretically possible to construct a mesoscopic rate such that the mean binding time in the two models match. Below that critical mesh size, no such rate can be constructed. Hence, in the presence of bimolecular reactions, a fundamental limitation of the RDME results from the inherent bound on the accuracy to which we can represent diffusion. As a direct consequence of our previous analysis, we also obtained mesoscopic reaction rates which ensure that the mean binding time in the RDME matches that of the Smoluchowski model for an irreversible bimolecular reaction.  

In this paper we extend our previous analysis. First, we study the case of reversible reactions, and ask whether there are additional constraints on the admissible mesh sizes compared to the irreversible case. We derive a critical voxel size under the following three assumptions: the average time until a reaction fires should match between mesoscopic and microscopic models, the steady state levels should match, and the mesoscopic dissociation rate should be smaller than or equal to the intrinsic dissociation rate. The last condition is necessary for the dissociation to make physical sense, and if we are to match not only equilibrium distributions but also transient solutions. The result establishes the previously obtained critical voxel sizes---in the perfectly absorbing, irreversible, case---as a fundamental lower limit for the more general reversible case. Importantly, this means that there will be a non-trivial lower limit for the mesh size, independent of the intrinsic microscopic reaction parameters. We also study the accuracy of our reaction rates as compared to the Smoluchowski model, and show good agreement, both at steady state and during the transient phase, between mesoscopic and microscopic simulations as the mesh size approaches the critical lower limit. In particular, we show how the multiscale propensities can provide a better mesoscopic approximation to a diffusion limited model of a MAPK cascade, compared to the widely used propensities of Collins and Kimball \cite{Kimball}. It has not previously been possible to accurately simulate this model with a fully local RDME, although it has been simulated successfully with a non-local extension of the RDME \cite{FBSE10}, and using a hybrid microscale-mesoscale method \cite{HeHeLo}.

For simplicity we have chosen not to write out units explicitly. We are using SI units throughout.

\section{Background}

A system of $N$ molecules of $S$ different chemical species, diffusing and reacting inside a finite reaction volume, can be modeled at several different scales, and the accuracy of the different models depends on the properties of the system. In this section we briefly review the microscopic Smoluchowski model and the mesoscopic RDME model, and discuss previous work on connecting the two models. 

\subsection{Microscopic scale: the Smoluchowski model}

Several microscopic models have been proposed and studied in some detail \cite{ErCha09,SamDoiSDLR,AnB}. We have chosen to focus on the Smoluchowski model, given the extensive attention it has received for instance in \cite{ZoWo5a,ZoWo5b,SHeLo11}. In the Smoluchowski framework, molecules are modeled as hard spheres. The radii of the spheres are referred to as the reaction radius of the molecules. The Smoluchowski equation, extended with a Robin boundary condition at the sum of the reaction radii, determines the probability of a reaction occurring between colliding molecules.

Let $\mathbf{x}_1$ and $\mathbf{x}_2$ be the positions of two molecules. Consider their relative position $\mathbf{r}=\mathbf{x}_1-\mathbf{x}_2$. The governing equation is given by
\begin{align}
\label{eq:smol1}
\frac{\partial \smolpdf}{\partial t} = D\Delta \smolpdffull,
\end{align}
where $\smolpdffull$ is the probability distribution function (PDF) of the relative position $\mathbf{r}$ at time $t$, given the relative position $\mathbf{r}_n$ at time $t_n$, and where $D$ is the sum of the diffusion constants of the molecules.  The boundary condition is given by
\begin{align}
\label{eq:smol2}
\smolboundaryshort,
\end{align}
where
\begin{align}
\label{eq:smol3}
\Kdef
\end{align}
Here $\rradius$ is the sum of the reaction radii of the molecules, and $\krmi$ is the microscopic, intrinsic, reaction rate. The initial condition is $\smolinit$.

To update a pair of molecules we solve for $\smolpdffull$, and then sample the new relative position at time $t$. Single molecules are updated by sampling from a normal distribution in all directions. The time until a dissociation is sampled from an exponential distribution, and the resulting products are placed at a distance equal to $\rradius$.  A system of $N$ molecules becomes an $N$-body problem, and a direct solution is generally unattainable. Instead we can simulate the system with the Green's function reaction dynamics (GFRD) algorithm \cite{ZoWo5a,ZoWo5b}. The core of the algorithm is a reduction of the problem to a collection of single- and two-body systems, accomplished through an appropriate restriction of the time step of the algorithm. In \cite{HeHeLo} the GFRD algorithm was extended to include complex boundaries, and in \cite{SHeLo11} improvements were suggested with the aim of making the algorithm more flexible and efficient. 

The GFRD algorithm is efficient for dilute systems where the free space between molecules allows for an efficient grouping in pairs while using a relatively large time step, and the computational benefit over brute-force BD methods can be orders of magnitude. If the system contains some species that are present in higher copy numbers, or if a very high spatial resolution is not required, the mesoscopic RDME can in turn be orders of magnitude faster than GFRD. 

\subsection{Mesoscopic scale: The reaction-diffusion master equation}

The RDME extends the classical well-mixed Markov-process model \cite{gillespie, GillespieRig} to the spatial case by introducing a discretization of the domain into $\numvox$ non-overlapping voxels \cite{Gardiner}. Molecules are point particles and the state of the system is the discrete number of molecules of each of the species in each of the voxels. A common choice for the discretization is a uniform Cartesian lattice, where each voxel is a square with area $h^2$ (2D), or a cube with volume $h^3$ (3D). Simulations can also be conducted on unstructured triangular and tetrahedral meshes for better geometric flexibility \cite{EnFeHeLo}. The RDME is the forward Kolmogorov equation, governing the time evolution of the probability density of the system.  

For brevity of notation, we write $p(\fatx,t) = p(\fatx,t|\fatx_0,t_0)$ for the probability that the system can be found in state $\fatx$ at time $t$, conditioned on the initial condition $\fatx_0$ at time $t_0$. For a general reaction-diffusion system, the RDME can be written as

\begin{align}
\label{eq:rdme}
\frac{\mathrm{d}}{\mathrm{dt}}p(\fatx, t) = 
&\sum_{i=1}^{K}
\sum_{r = 1}^{M}
a_{ir}(\fatx_{i \cdot}-\fatmu_{ir})p(\fatx_{1 \cdot},\ldots,\fatx_{i \cdot}-\fatmu_{ir},
\ldots,\fatx_{K \cdot}, t) \nonumber 
-\sum_{i=1}^{K}
\sum_{r = 1}^{M}
a_{ir}(\fatx_{i \cdot})p(\fatx, t)\\
&+\sum_{j=1}^{N} \sum_{i = 1}^{K} \sum_{k=1}^K d_{jik}(\fatx_{\cdot j}-\fatnu_{ijk})
p(\fatx_{\cdot 1},\ldots,\fatx_{\cdot j}-\fatnu_{ijk},
\ldots,\fatx_{\cdot N}, t) \nonumber\\
&-\sum_{j=1}^N\sum_{i=1}^{K}
\sum_{k = 1}^{K} d_{ijk}(\fatx_{\cdot j})p(\fatx, t),\nonumber\\
\vspace{-30pt}
\end{align}
\noindent
where $\fatx_{i\cdot}$ denotes the $i$-th row and $\fatx_{\cdot j}$ denotes the $j$-th column of the $K\times S$ state matrix $\fatx$ where $S$ is the number of chemical species. The functions $a_{ir}(\fatx_i)$ define the propensities of the $M$ chemical reactions, $\fatmu_{ir}$ are stoichiometry vectors associated with the reactions.  $a_{ir}(\fatx)$ and $\fatmu_{ir}$ have the same meaning as for a well mixed system but are now defined for each of the voxels. $d_{ijk}(\fatx_i)$ are propensities for the diffusion jump events, and $\fatnu_{ijk}$ are stoichiometry vectors for diffusion events. $\fatnu_{ijk}$ has only two non-zero entries, corresponding to the removal of one molecule of species $X_k$ in voxel $i$ and the addition of a molecule in voxel $j$. The RDME is too high-dimensional to permit a direct solution. Instead realizations of the stochastic process are  sampled, using algorithms similar to the SSA but optimized for reaction-diffusion systems \cite{ElEh04}.

The propensity functions for the diffusion jumps, $d_{ijk}$, are selected to provide a consistent and local discretization of the diffusion equation, or equivalently the Fokker-Planck equation for Brownian motion \cite{GillHellPetz}. For a uniform Cartesian grid, a finite difference discretization results in diffusion jumps with propensities 
\begin{align}
X_{is} \xrightarrow{d_{ijs}} X_{js}, \quad d_{ijs} = \frac{\gamma_s}{h^2},
\end{align}
\noindent
where $\gamma_s$ is the diffusion coefficient of $X_s$ and $d_{ijs}$ is non-zero only for adjacent voxels. 
For a triangular or tetrahedral unstructured mesh, finite element or finite volume discretizations result in propensities that account for the shape and size of each voxel \cite{EnFeHeLo}. With this model of diffusion, setting reactions aside, the solution of the RDME will, with vanishing voxel sizes, converge in probability to Brownian motion. 

In the case of mass action kinetics, the propensity functions $a_{ir}(\fatx)$ for bimolecular reactions take the form
\begin{align}
a_{ir}(\fatx_i) = \krme x_{is}x_{is^{'}}.
\end{align}
We will refer to $\krme$ as the \emph{mesoscopic association rate}. In practical modeling work, $\krme$ is often obtained by fitting the model to some phenotypic experimental observation, or provided directly as a macroscopic or mesoscopic model parameter obtained from experiments. If the association rate is instead given in terms of the microscopic reaction rate, $\krmi$, a result of Collins and Kimball \cite{Kimball} provides the effective rate, $\kme$, of the system, which can be used for mesoscale simulations in 3D through $\krme=\kme/h^3$, where
\begin{align}
\kme = \frac{4\pi\rradius\diffconst\krmi}{4\pi\rradius\diffconst+\krmi}. 
\label{eq:ck}
\end{align} 

Gillespie derives the same relation by applying classical results from gas kinetics, and calls it the diffusional propensity function \cite{gillespie:164109}. From Gillespie's physically rigorous derivation it is possible to relate the intrinsic reaction rate $\krmi$ to fundamental physical constants.  

A natural question, given model parameters and a mesh size, is how well a mesoscopic simulation can capture the microscale dynamics. It has been shown that the solution of the RDME diverges with respect to the Smoluchowski model \cite {Isaacson,HHP}. Due to the point particle assumption, bimolecular reactions occur with successively lower probability, to eventually vanish in the limit of small voxels.

This means that there is a lower bound on the voxel size, below which bimolecular reactions cannot be accurately simulated in the RDME. There is also an upper bound on the voxel size, above which the reaction-diffusion dynamics will be insufficiently resolved. The question of how to choose the voxel size for sufficient accuracy has not yet been satisfactorily answered, but some attempts at establishing lower bounds on the voxel size have been made. A trivial bound on the voxel size follows from physical arguments \cite{GPS:2014}; the voxels must be large enough for the reacting molecules to remain dilute and well mixed inside the voxel. This condition translates to
\begin{align}
h \gg \rradius.
\label{eq:trivialcondition}
\end{align}

Apart from physical common sense, the above condition is an explicit assumption in the derivation of \eqref{eq:ck}. Only if this assumption is valid can the rate constant \eqref{eq:ck} be expected to provide an accurate simulation with respect to the Smoluchowski model. Unfortunately, condition \eqref{eq:trivialcondition} offers little guidance on how to choose the mesh size in practice. 

The reaction constant \eqref{eq:ck} depends on the microscopic parameters but not on the spatial discretization. By allowing the rate to depend on the mesh, and by matching certain properties of the microscopic model, it is possible to derive alternative forms for $\krme$ that perform better than \eqref{eq:ck} for diffusion limited systems and for fine meshes \cite{FBSE10, HHP}. Those approaches also yield constants in 2D, where no expression based on a physical derivation is available. In \cite{FBSE10} the rates are derived by matching the mean equilibration time of a reversible reaction on a spherical discretization. These rates are then used on Cartesian lattices.

Let $\micromrtd(\krmi)$ be the mean binding time for two molecules in the Smoluchowski model in dimension $d$. In the case of a single irreversible association reaction on a uniform Cartesian discretization of a square or cube of side length $L$, Hellander et al. \cite{HHP} showed that if the mesoscopic propensity $\krme$ is chosen as $\krme=\hhlrated$, where
\begin{align}
\hhlratedfull=\left\{\begin{array}{ll}
\frac{\left(L/h\right)^2}{\micromrtd(\krmi)-[\frac{L^2}{2\pi D}\log\left(\frac{L}{h}\right)+\frac{0.1951L^2}{4D}]} & d=2\\
\frac{(L/h)^3}{\micromrtd(\krmi)-1.5164L^3/(6Dh)} & d=3,
\end{array}\right.
\label{eq:ourcorr}
\end{align}
\noindent 
then the mean binding time on the mesoscopic scale will match the mean binding time in the Smoluchowski model. For $k_r\to\infty$, this is possible down to $\hstarinf\approx 3.2\rradius$ in 3D, and $\hstarinf\approx 5.1\rradius$ in 2D. Below $\hstarinf$, the matching of the mean binding time will not be possible in the perfectly absorbing case. For simple geometries, such as a disk or a sphere, $\micromrtd$ can be obtained analytically. For other geometries, provided that $h\ll L$, the analytical result for a disk or a sphere with matching volume provides an excellent approximation. The analytical lower bounds on the voxel sizes are obtained by considering the extreme case of $\krmi \to \infty$ and using the above mentioned approximation for $\micromrtd$. In \cite{ErCha09}, based on another assumption not involving the microscopic Smoluchowski model, Erban and Chapman arrived at an expression in 3D that is similar to \eqref{eq:ourcorr} and that establishes the same critical mesh size. 

In what follows, we set out to expand on this theory with the aim of an improved understanding of the range of the mesh sizes for which the RDME will accurately approximate the Smoluchowski model. In particular, we derive critical mesh sizes for more realistic kinetics like reversible bimolecular reactions, and we show that under mild assumptions, the rates \eqref{eq:ourcorr} are effectively independent of $L$. We also obtain error estimates that provide a way to estimate the needed mesh resolution to achieve a certain accuracy in the rebinding time distributions.

\section{Results}

\label{sec:difftotarget}

The case of a single irreversible reaction was studied in \cite{HHP}. The analysis provided analytical lower bounds on the mesh size in 2D and 3D only for the case of $\krmi\to\infty$, and reaction rates for matching mean association times. The reaction rates depended on the size $L$ of the domain. As reactions occur locally in space, that dependence is not intuitive. 

In this section we expand on that theory. First we show that the reaction rates are independent of the size of the domain, under the assumption that the domain is much larger than the molecules. From this follows analytical lower bounds on the mesh size for the case of an irreversible reaction with $0<\krmi\leq \infty$.

Second, we study a reversible reaction on a square or cubical domain and proceed to derive mesoscopic reaction rates under the following three assumptions: the average reaction time should match between mesoscopic and microscopic models, the steady state levels should match, and the intrinsic dissociation rate should be larger than or equal to the mesoscopic dissociation rate. Under these assumptions we derive a lower bound on the mesh size independent of the reaction rates.

We also show how the multiscale mesoscopic reaction rates behave in the limit of small voxels as well as in the limit of large voxels, effectively connecting the microscopic and mesoscopic scales.  Finally, we provide error estimates that relate the mesh size to the error in rebinding-time distributions, and show that the mean mesoscopic rebinding time approaches the mean microscopic rebinding time, as the mesoscopic mesh size approaches the lower bound.

\subsection{Irreversible reactions}
\label{sec:local_irreversible}
The analytical expressions for the lower bounds on the voxel size $h$ derived in \cite{HHP} are valid for the case of irreversible reactions with perfect absorption. In this section we assume $L\gg\sigma$, to obtain analytical expressions for the lower bounds in the general case of $k_r>0$. We show that the reaction rates are independent of the size $L$ of the domain, thus depending on local parameters only.

\begin{theorem}
Let $\hhlrated$ be the mesoscopic reaction rate in dimension $d$, and assume that $L\gg\sigma$. Then
\begin{align}
\label{eq:hhlrate}
\hhlratedfull \approx \frac{\krmi}{h^d}\left( 1+\frac{\krmi}{D}\Gd (h,\sigma) \right)^{-1}
\end{align}
where $d\in \{2,3\}$, and
\begin{align}
\Gdfull = \begin{cases}
\frac{1}{2\pi}\log\left(\pi^{-\frac{1}{2}}\frac{h}{\sigma}\right)-\frac{1}{4}\left(\frac{3}{2\pi}+\Calphatwo\right)\quad (2D)\\
\frac{1}{4\pi\sigma}-\frac{\Calphathree}{6h}\quad (3D),
\label{eq:gdfull}
\end{cases}
\end{align}
where
\begin{align}
\Calphad \approx \begin{cases}
0.1951,\quad d=2\\
1.5164, \quad d=3.
\end{cases}
\end{align}
Let 
\begin{align}
\label{eq:hstarirr_def}
\hstarirr=\inf_{h}\{\hhlrated(\krmi,h)>0 \}. 
\end{align}
Then $\hhlrated$ is a well-defined reaction rate only for $\voxsize > \hstarirr$, and
\begin{align}
\label{eq:hstarrevexpr}
\hstarirr = \begin{cases}
\sqrt{\pi}e^{\frac{3+2\Calphatwo\pi}{4}-\frac{2\pi D}{\krmi}}\rradius\quad (2D)\\
\frac{\Calphathree}{6}\left( \frac{D}{\krmi}+\frac{1}{4\pi\sigma}\right)^{-1} (3D).
\end{cases}
\end{align}

\end{theorem}

\begin{proof}
In \cite{HHP}, the reaction rate $\hhlrated$ was derived under the assumption that $L\gg\voxsize$. This effectively implies that $L\gg\sigma$, since we have adopted the basic assumption \eqref{eq:trivialcondition}. Let $\micromrtd$ and $\mesomrtd$ be the mean association times for uniformly distributed particles in dimension $d$. While individual voxels may be too small for the Collins and Kimball approximation to be valid,  the system as a whole may be well-mixed and dilute. Thus, if we assume that $L\gg\sigma$ we can approximate the global mean reaction time $\micromrtthree$ by
\begin{align*}
\micromrtthree \approx \frac{L^3}{\kme},
\end{align*}
which, when inserted into \eqref{eq:ourcorr}, yields
\begin{align*}
\hhlratethreefull = \frac{h^{-3}}{\kme^{-1}-\frac{\Calphathree}{6Dh}}.
\end{align*}
With $\kme$ defined by \eqref{eq:ck}, we obtain $\hhlratethree$ as in \eqref{eq:hhlrate}. Thus, for large enough domains, the effect of the outer boundary is small in 3D, and the reaction rates are defined by local parameters only. 

The situation in 2D is more complicated. In \cite{FBSE10} they derive $\micromrttwo$ for a disk. Assuming $L\gg\sigma$, this will provide an excellent estimate of $\micromrttwo$ in the case of a square. Thus
\begin{align*}
\micromrttwo \approx \frac{[1+\alpha F(\lambda)]L^2}{\krmi},
\end{align*}
where
\begin{align*}
\begin{cases}
\lambda = \pi^{\frac{1}{2}}\frac{\rradius}{L}\\
\alpha = \frac{\krmi}{2\pi D} \\
F(\lambda) = \frac{\log(1/\lambda)}{(1-\lambda^2)^2}-\frac{3-\lambda^2}{4(1-\lambda^2)},
\end{cases}
\end{align*}
yields
\begin{align*}
\hhlratetwofull &\approx \frac{\voxsize^{-2}}{\frac{1+\alpha F(\lambda)}{\krmi}-\left( \frac{1}{2\pi \diffconst}\log(\frac{L}{\voxsize})+\frac{\Calphatwo}{4\diffconst} \right)} \\ &=
\frac{1}{\voxsize^2}\left( (\krmi)^{-1}+ \frac{1}{2\pi \diffconst}\left\{\frac{\log(\lambda^{-1})}{(1-\lambda^2)^2}-\frac{(3-\lambda^2)}{4(1-\lambda^2)^2}\right\}-\frac{1}{2\pi \diffconst}\log\left(\frac{L}{\voxsize}\right)- \frac{\Calphatwo}{4\diffconst} \right)^{-1},
\end{align*}

For $L\gg\rradius$ we have $1-\lambda^2\approx 1$, and $3-\lambda^2\approx 3$. We then get
\begin{align}
\label{eq:hhl_2D_eq}
\hhlratetwofull \approx \frac{1}{\voxsize^2}\left( (\krmi)^{-1}+\frac{1}{2\pi \diffconst} \left\{ \log \left( \pi^{-\frac{1}{2}}\frac{L}{\rradius}\right)-\log\left( \frac{L}{\voxsize} \right) \right\}-\frac{1}{4\diffconst}\left(\frac{3}{2\pi}+\Calphatwo\right) \right)^{-1}.
\end{align}
By noting that
\begin{align*}
\log\left( \pi^{-\frac{1}{2}}\frac{L}{\sigma}\right)-\log\left(\frac{L}{h}\right) = \log\left(\pi^{-\frac{1}{2}}\frac{h}{\sigma}\right),
\end{align*}
we find that we can rewrite \eqref{eq:hhl_2D_eq} to obtain \eqref{eq:hhlrate} also in the case of $d=2$. Consequently, for $L\gg\sigma$, the reaction rate defined by $\hhlratetwo$ is approximately independent of the global parameter $L$. 

Now \eqref{eq:hstarrevexpr} follows by noting that $\hhlrated>0$ holds if and only if
\begin{align}
\label{eq:rate_cond1}
\frac{\krmi}{\diffconst}\Gdfull>-1,
\end{align}
and then solving $\frac{\krmi}{\diffconst}\Gdfull=-1$ for $h$.
\end{proof}

\subsection{Reversible reactions}
\label{sec:local}

In this section we extend the analysis to the reversible case. We limit our considerations to the conventional RDME, thus allowing reactions only between molecules occupying the same voxel. The system consists of one $A$-molecule and one $B$-molecule, diffusing on a Cartesian lattice consisting of $N$ voxels of width $h$, with periodic boundary conditions. The molecules react reversibly through the reaction
\begin{align}
\label{revreaction}
A+B \overset{\krme}{\underset{\kdme}{\rightleftharpoons}} C.
\end{align}

\noindent 
Henceforth we assume that $\krmi$, the microscopic association rate, and $\kdmi$, the microscopic dissociation rate, are given model parameters, and that $\krme$ is defined by \eqref{eq:hhlrate}.

It remains to define $\kdme$.  A plausible approach would be to match the steady-state levels of the species at the mesoscopic and microscopic scales. For a system of one $A$- and one $B$-molecule, the steady state is governed by the ratio of the average time the molecules are unbound to the total time. Thus, to match the steady state at the different scales, we choose $\kdme$ such that
\begin{align}
\label{eq:steadystate}
\frac{\tau_{rebind}^{meso}}{\tau_d^{meso}+\tau_{rebind}^{meso}} = \frac{\tau_{rebind}^{micro}}{\tau_d^{micro}+\tau_{rebind}^{micro}},
\end{align}
where $\tau_{rebind}^{meso}$ and $\tau_{rebind}^{micro}$ are the respective times until an $A$- and a $B$-molecule rebind, given that they have just dissociated. The quantities $\tau_d^{meso}$ and $\tau_d^{micro}$ are the average times until a $C$-molecule dissociates (thus $\tau_d^{meso}=1/k_d^{meso}$ and $\tau_d^{micro} = 1/k_d$).

We can compute $\tau_{rebind}^{meso}$ analytically by expressing it in terms of the mean binding time $\mesomrtd$. We note that the rebinding time is given by 
\begin{align}
\label{eq:meso_rebind}
\tau_{rebind}^{meso}(\krme) = \mesomrtd(\krme)-\mesomrtd(\infty).
\end{align}
When the molecules have reached the same voxel, the system is in the same state as immediately following a dissociation. Thus, by subtracting from the total mean binding time the time it takes to reach the same voxel, we are left with the average rebinding time.


Expressions for $\mesomrtd$ are given in \cite{HHP}, and inserting them into \eqref{eq:meso_rebind} yields
\begin{align}
\label{eq:tau_rebind_solved}
\tau_{rebind}^{meso} = \frac{N}{\krme}.
\end{align}

Unfortunately $\tau_{rebind}^{micro}$ is not easily computed by analytical means (for general geometries), but by noting that similarly as in the mesoscopic case, $\tau_{rebind}^{micro}(k_r)=\micromrtd(k_r)-\micromrtd(\infty)$ we obtain the estimate
\begin{align}
\label{eq:micro_estimate}
\tau_{rebind}^{micro} \approx \frac{L^d}{\krmi}
\end{align}
for $L\gg\sigma$. The argument is the same as for the mesoscopic case; $\micromrtd(\infty)$ represents the time until two molecules are in contact for the first time. By subtracting that time from the total time for the two molecules to react given a uniform initial distribution, we are left with the rebinding time. By using \eqref{eq:tau_rebind_solved} and \eqref{eq:micro_estimate} in  \eqref{eq:steadystate} it follows that
\begin{align}
\frac{\frac{N}{\krme}}{\frac{1}{\kdme}+\frac{N}{\krme}} = \frac{\frac{L^d}{\krmi}}{\frac{1}{\kdmi}+\frac{L^d}{\krmi}},
\end{align}
which, after some simplifications, results in 
\begin{align}
\label{eq:meso_k_d}
\kdme = \voxsize^d\frac{\kdmi\krme}{\krmi}.
\end{align}
When rearranged, \eqref{eq:meso_k_d} can be recognized as the detailed balance condition. Maintaining this relation between the rate constants is sufficient to ensure that the equilibrium values of $A$ and $B$ are the same in the two models. 

However, though we can match the mean association time as well as the equilibrium values for a given voxel size, we may have a mesoscopic dissociation rate that is faster than the microscopic dissociation rate. That would result in more dissociation events on the microscopic scale, and since the microscopic scale is assumed to be more fine-grained, this is unphysical. This can be seen by considering that the inverse of the mesoscopic dissociation rate is a combination of the expected time for a complex to break apart and the time for the products to become well-mixed by diffusion inside a voxel. The mesoscopic dissociation rate thus includes the possibility of fast microscopic rebinding events that occur before the molecules have become well-mixed. Thus, while matching the mean association time and the equilibrium values of the molecules, we may end up with incorrect transient dynamics due to too fast dissociations on the mesoscopic scale. 
The intrinsic microscopic dissociation rate is therefore an upper bound on the mesoscopic rate. As a consequence we want to determine for what size of $\voxsize$ we can match the mean binding times while satisfying
\begin{align}
\label{eq:diss_cond}
\kdme \leq \kdmi.
\end{align}

\begin{theorem}
The condition $\kdme \leq \kdmi$ is satisfied if and only if the inequality
\begin{align}
\label{eq:G_ineq2}
\Gdfull>0
\end{align}
holds for $\Gdfull$ as defined in \eqref{eq:gdfull} . 

Let 
\begin{align*}
\hstarrev= \max\left\{ \hstarirr, \inf_h\left\{ \kdme \leq \kdmi \right\} \right\}. 
\end{align*}
Then $\hstarrev$ is the smallest $h$ for which we can satisfy $\micromrtd=\mesomrtd$ as well as $\kdme\leq\krme$, and we have
\begin{align}
\label{eq:th2_2}
\hstarrev = \hstarinf \approx \begin{cases}
5.1\sigma \quad (2D)\\
3.2\sigma \quad (3D).
\end{cases}
\end{align}
\end{theorem}
\begin{proof}
We obtain $\hstarrev$ by using \eqref{eq:meso_k_d} in \eqref{eq:diss_cond}:
\begin{align}
h^d\frac{k_d^{micro}k_r^{meso}}{k_r} \leq k_d,
\end{align}
which holds if and only if
\begin{align}
\label{eq:k_d_ineq}
\frac{h^d}{\krmi}\krme\leq 1.
\end{align}
However, from \eqref{eq:hhlrate} we have
\begin{align*}
\frac{h^d}{\krmi}\krme = \left(1+\frac{\krmi}{D}\Gdfull\right)^{-1}.
\end{align*}
Thus we find that \eqref{eq:k_d_ineq} is equivalent to
\begin{align}
\label{eq:G_cond}
\frac{\krmi}{D}\Gdfull > 0,
\end{align}
since $(\krmi/\diffconst)\Gdfull < -1$ is excluded by the condition $\hhlrated > 0$ as shown in \eqref{eq:rate_cond1}. Since $\krmi\geq 0$ and $\diffconst>0$ (we do not consider the case $\diffconst=0$ here), \eqref{eq:G_ineq2} follows. 

We now observe that $\hstarirr\leq \hstarinf$, and that, immediately from \eqref{eq:hhlrate} and \eqref{eq:hstarirr_def}, it follows that $\hstarinf = \inf_h\left\{ \Gdfull>0 \right\} = \inf_h\left\{ \kdme \leq \kdmi \right\}$. The last equality is an immediate consequence of \eqref{eq:G_ineq2}, and we obtain \eqref{eq:th2_2}.

\end{proof}
In Figure \ref{fig:hstar_vs_hstarinf} we show how $\hstarrev$ relates to $\hstarirr$.

\begin{figure}

\centering
\includegraphics[width=0.8\linewidth]{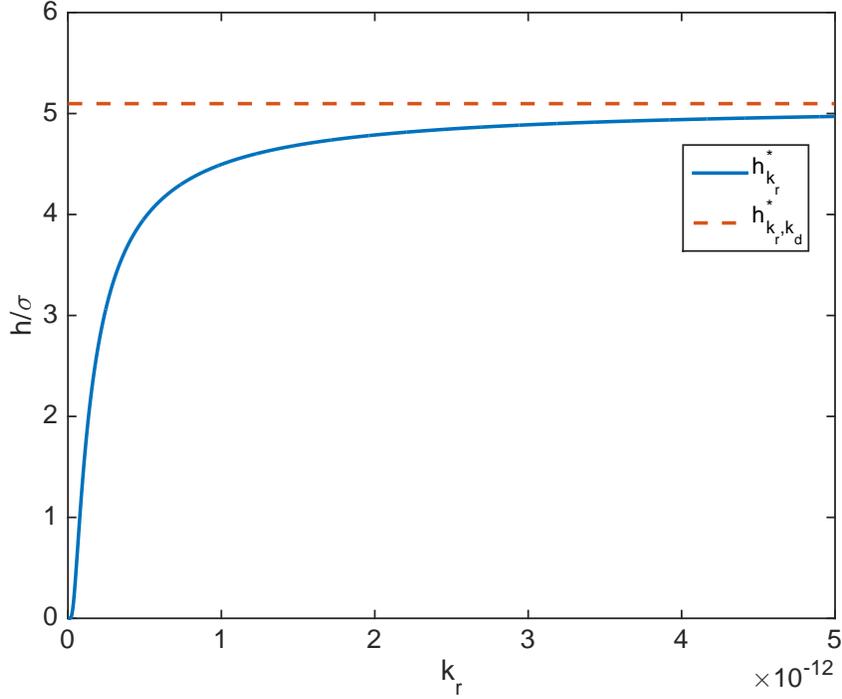}
\caption{\label{fig:hstar_vs_hstarinf}The bound $\hstarirr$ for the irreversible case in 2D depends on $\krmi$, and decreases with decreasing $\krmi$. However, in order to satisfy \eqref{eq:diss_cond}, we must have $h>\hstarrev$, shown as the red dashed line in the figure above. For large $\krmi$, $\hstarirr\approx\hstarrev$. By \eqref{eq:th2_2}, $\hstarrev = \hstarinf$ for all $\krmi$, $\kdmi>0$. The remaining parameters are given by $\sigma=2\cdot 10^{-9}$ and $D = 2\cdot 10^{-14}$.}
\end{figure}

\subsection{Some limit cases}
\label{sec:limits}
In this section we investigate the behavior of the reaction rate $\hhlratedfull$ in some limit cases. We would expect $\hhlratethree$ to behave similarly to $\kme$ for large voxels. In the limit of small voxels, it is of interest to see how the mesoscopic reaction rates relate to the microscopic reaction rates.

\begin{corollary}
For $L/\sigma\gg 1$ we have
\begin{align}
\label{eq:limit1}
\hhlrated(\krmi,\hstarinf) \approx \frac{\krmi}{(\hstarinf)^d},
\end{align}
Also, as $\voxsize\to\infty$, with $L/\voxsize$ constant and $L/\voxsize\gg 1$, we have
\begin{align}
\label{eq:limit2}
\hhlratethreefull \to \frac{\kme}{h^3}.
\end{align}
In the limit of $\krmi/D\to 0$ we obtain
\begin{align}
\label{eq:limitsmallkr}
\hhlratedfull \to \frac{\krmi}{\voxsize^d}.
\end{align}
In 3D, \eqref{eq:limitsmallkr} implies that
\begin{align}
\label{eq:limit4}
\hhlratethreefull\to\frac{\kme}{\voxsize^3}.
\end{align}
as $\krmi/D\to 0$.
\label{coll:limits}
\end{corollary}

\begin{proof}
Since $\Gd(\rradius,\hstarinf)=0$, \eqref{eq:limit1} follows immediately. By noting that 
\begin{align*}
\Gthreefull \to \frac{1}{4\pi\sigma}\,\, , as\, \voxsize\to\infty,
\end{align*}
we get \eqref{eq:limit2} from some straightforward algebra. We obtain \eqref{eq:limitsmallkr} immediately from \eqref{eq:hhlrate}. Finally, \eqref{eq:limit4} follows from the fact that $\kme\to\krmi$ as $\krmi\to 0$ or $\diffconst\to\infty$.
\end{proof}

In words, as we approach the critical mesh size $\hstarinf$ in the mesoscopic model, the mesoscopic rates approach the microscopic, intrinsic rates. For large voxels in 3D, we note that $\hhlratethree$, as expected, approaches $\kme$.

As an immediate consequence we have
\begin{corollary}
For $L/\rradius\gg 1$ and $h=\hstarinf$, the following holds:
\begin{align}
\begin{cases}
\hhlrated(\krmi,\hstarinf) \approx \frac{\krmi}{(\hstarinf)^d}\\
\kdme \approx \kdmi\\
\tau_{\rm rebind}^{meso} \approx \tau_{\rm rebind}^{\rm micro}.
\end{cases}
\end{align}
\end{corollary}
\begin{proof}
The first approximation is proven in Corollary \ref{coll:limits}. Using $\hhlrated(\krmi,\hstarinf) \approx \frac{\krmi}{(\hstarinf)^d}$ in \eqref{eq:meso_k_d} we obtain $\kdme \approx \kdmi$. From \eqref{eq:tau_rebind_solved} and \eqref{eq:micro_estimate} it then follows that $\tau_{\rm rebind}^{meso} \approx \tau_{\rm rebind}^{\rm micro}$.
\end{proof}
For $\voxsize=\hstarinf$, we match the mean association time, dissociation time, and the mean rebinding time. For any $\voxsize$ below $\hstarinf$, we cannot simultaneously match both the association and dissociation times. Consequently the rebinding dynamics will be less accurately captured on the mesoscopic scale for $\voxsize<\hstarinf$ than for $\voxsize=\hstarinf$.

In Figure~\ref{fig:limits3D} we illustrate these limits in 3D for different values of the intrinsic reaction rate $\krmi$. The more diffusion limited the reaction is, the larger the discrepancy between $\hhlratethree$ and $\kme$ becomes.

\begin{figure}[htp]
\centering
\subfigure[$k_r = 10^{-18}$]{\includegraphics[width=0.75\linewidth]{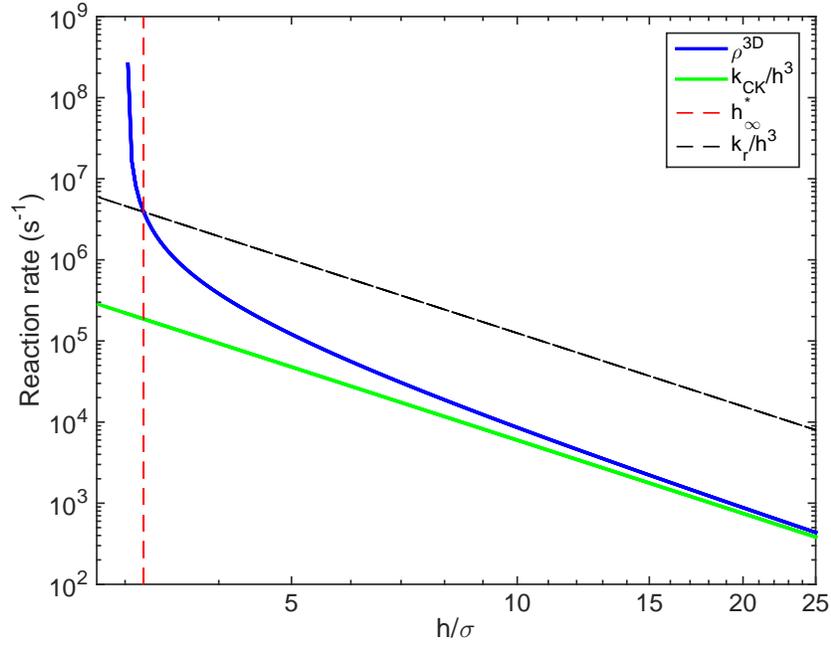}}
\subfigure[$k_r = 10^{-20}$]{\includegraphics[width=0.75\linewidth]{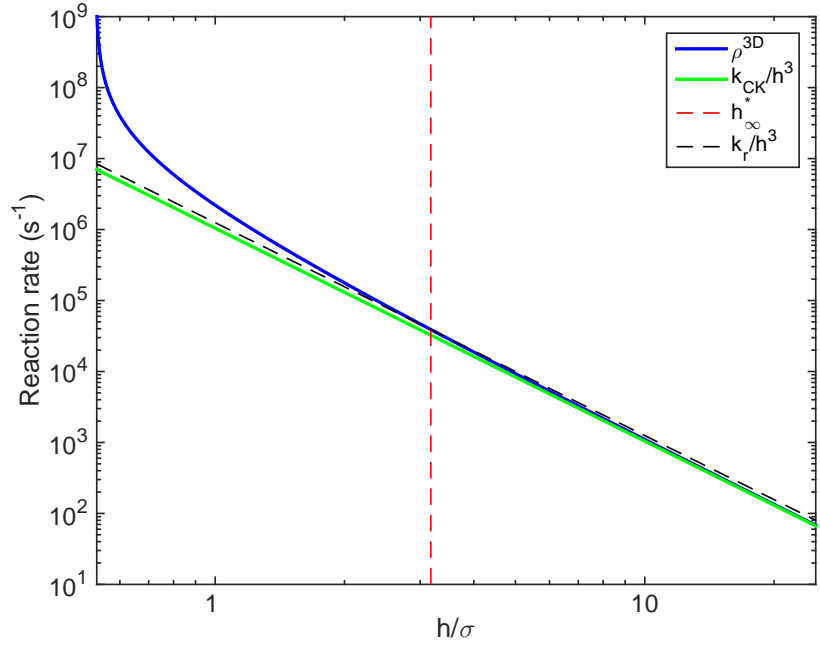}}
\caption{\label{fig:limits3D}Limits in 3D. In (a), where $k_r=10^{-18}$, the difference between $\kme$ and $\hhlrated$ is orders of magnitude; in (b), where $k_r=10^{-20}$, the difference is much smaller. The other parameters are given by $\diffconst = 2\cdot 10^{-12}$, $\rradius = 2\cdot 10^{-9}$, and $L=5.145\cdot 10^{-7}$.}  
\end{figure}

\subsection{Error estimates}

Given a system with known intrinsic reaction rates, we would ideally want to know how to choose the mesh size for sufficiently accurate mesoscopic simulations. While solving this is hard for a general system, we can choose the mesh size such that we limit the relative error in the average rebinding times. We do not guarantee an accurate mesoscopic solution by doing so, but if the error in the average rebinding time is large, and we have fine-grained dynamics that we wish to capture, we may have to reduce the voxel size to decrease the error.

Now consider a reversible reaction, with intrinsic reaction rates given by $\krmi$ and $\kdmi$.
\begin{corollary}
For a given error tolerance $\epsilon$, the following holds:
\begin{align}
|\krmi-h^d\hhlrated |<\epsilon\krmi
\end{align}
if and only if
\begin{align}
\voxsize\leq F(\krmi,\rradius,\diffconst,\epsilon),
\end{align}
where
\begin{align}
\label{eq:F_2D_3D}
F(\krmi,\rradius,\diffconst,\epsilon) = \begin{cases}
\frac{\Calphathree}{6}\left[\frac{1}{4\pi\sigma}+(1-(1-\epsilon)^{-1})\frac{D}{\krmi} \right]^{-1}\quad (3D)\\
\sqrt{\pi}\exp\left[ -\frac{2\pi\diffconst}{\krmi}\left(1-(1-\epsilon)^{-1}\right)+\frac{3+2\pi\Calphatwo}{4}\right]\sigma.\quad (2D)
\end{cases}
\end{align}
Furthermore, for $h\leq F$, we have
\begin{align}
\label{eq:rebind_ineq}
\frac{|\tau_{\rm rebind}^{meso}-\tau_{\rm rebind}^{micro}|}{\tau_{\rm rebind}^{micro}} \leq \epsilon+O(\epsilon^2).
\end{align}

In 2D we have 
\begin{align}
\label{eq:2D_limit}
F\to\infty, \text{ as } D/\krmi\to\infty.
\end{align}
 In 3D, we obtain
\begin{align}
\label{eq:F_limit_3D}
F\to\infty
\end{align}
as
\begin{align}
\label{eq:F_limit_3D_conds}
\begin{cases}
\frac{D}{\krmi} \to \frac{1}{4\pi\rradius\epsilon}, \text{ or}\\
\epsilon\to \frac{\kme}{4\pi\rradius\diffconst} =: \epsilon_{\rm max}
\end{cases}
\end{align}
For all $h$, it holds that
\begin{align}
\label{eq:upper_bound_error}
\frac{|\tau_{\rm rebind}^{meso}-\tau_{\rm rebind}^{micro}|}{\tau_{\rm rebind}^{micro}} < \frac{\kme}{4\pi\sigma D}+O\left( \left(\frac{\kme}{4\pi\sigma D}\right)^2 \right),
\end{align}
in 3D.
\end{corollary}
\begin{proof}
Assume that
\begin{align}
\label{eq:eps_assumption}
\left|\krmi-h^d\hhlrated\right|<\epsilon\krmi
\end{align}
holds. Using \eqref{eq:hhlrate} we obtain
\begin{align}
\label{eq:error1}
\left| 1-\left(1+\frac{\krmi}{D}\Gd \right)^{-1} \right|<\epsilon.
\end{align}
From \eqref{eq:G_cond} we know that $(\krmi/\diffconst) \Gd<0$, and thus \eqref{eq:error1} becomes
\begin{align}
\label{eq:no_abs}
1-\left(1+\frac{\krmi}{D}\Gd \right)^{-1} <\epsilon.
\end{align}
Some straightforward algebra now yields
\begin{align}
\label{eq:G_ineq}
\Gd<-\frac{D}{\krmi}\left( 1-(1-\epsilon)^{-1}\right).
\end{align}
By inserting $\Gd$, for $d=2$ and $d=3$, into \eqref{eq:G_ineq} and solving for $h$, we obtain \eqref{eq:F_2D_3D}.

We obtain \eqref{eq:rebind_ineq} by noting that $\krmi \approx L^d/\tau_{rebind}^{micro}$ and $\krme = N/\tau_{rebind}^{meso}$, and using that in \eqref{eq:eps_assumption}, together with the observation that $(1-(1-\epsilon)^{-1}) = -(\epsilon+O(\epsilon^2))$. 

As an immediate consequence of \eqref{eq:F_2D_3D} we obtain \eqref{eq:2D_limit}. We get \eqref{eq:F_limit_3D}-\eqref{eq:F_limit_3D_conds} from \eqref{eq:F_2D_3D}, and by noting that
\begin{align}
\frac{1}{4\pi\sigma}+(1-(1-\epsilon)^{-1})\frac{D}{\krmi}\to 0
\end{align}
as
\begin{align}
\epsilon\to \frac{\kme}{4\pi\rradius\diffconst}.
\end{align}
Finally we get \eqref{eq:upper_bound_error} immediately from \eqref{eq:rebind_ineq} and \eqref{eq:F_limit_3D}-\eqref{eq:F_limit_3D_conds}.
\end{proof}

We see that as the reactions become very diffusion limited, the difference between the mesoscopic and microscopic reaction rates can grow large, since $\kme\to 4\pi\sigma D$ as $\krmi\to\infty$. The less diffusion limited a reaction is, the closer the reaction rates will be (and the difference is bounded). This makes intuitive sense, since less diffusion-limited reactions mean that the system is more well-mixed; thus the system can be accurately simulated on a coarser mesh.

In Section \ref{sec:MAPK_example} we demonstrate how this theory can be applied to increase the understanding of the behavior of a relevant biological system.

\section{Numerical experiments}

In this section we present two numerical examples that will demonstrate the scope of validity of the mesoscopic reaction rates derived above. In the first example we consider the rebinding time of a pair of molecules. We compute the distributions and compare mesoscopic results for varying $h$ to microscopic simulations.

In the second example we study a model proposed by Takahashi et al. in \cite{TaTNWo10}. It was shown to have fine-grained dynamics, captured at the microscale but not at the deterministic level; we will simulate it at the mesoscopic scale, and show that as the mesh size $\voxsize$ approaches $\hstarinf$, the mesoscopic and microscopic scales agree.

\subsection{Rebinding-time distribution}

Consider one molecule of species $A$ and one molecule of species $B$, subject to a reversible reaction
\begin{align} 
A+B \underset{\kdmi}{\overset{\krmi}{\rightleftharpoons}} C.
\label{eq:reversible}
\end{align}
in a cubic domain of width $L$.

We showed in Section \ref{sec:limits} that for $h=\hstarinf$ the mean rebinding time at the mesoscopic scale will agree with the mean rebinding time at the microscopic scale. This does not, however, automatically guarantee that this particular choice of mesh size will yield the best agreement between the distributions of the rebinding times at the different scales.

In Figure~\ref{fig:rebind-dist-3D-2D} we plot the distribution of the rebinding times for different mesh sizes, and compare to the distribution obtained by simulations at the microscopic scale. For $\voxsize=\hstarinf$ we get a distribution that matches the microscopic results well for $t\gtrsim h^2/(2D)$. For other values of $h$, we get a distribution shifted relative to the microscopic distribution. For $t\lesssim h^2/(2D)$, the microscopic simulations behave differently than the mesoscopic ones, regardless of mesh size. This is due to the fact that at this timescale the spatial resolution is coarser than the temporal resolution, thus being the limiting factor for the accuracy. During a time $t$, a molecule diffuses an average distance proportional to $\sqrt{2Dt}$ in each direction. Thus, for $t\lesssim h^2/(2D)$, the distance the molecule diffuses is less than the size $\voxsize$ of a voxel.

On the mesoscopic scale, the distribution of the rebinding time is approximately exponential at short time scales (time scales smaller than the time it takes for a molecule to diffuse on average the distance of a voxel.) That gives the first plateau in Figure 3. At longer time scales the diffusion causes the rebinding time not to be exponential. In this region the mesoscopic simulations behave similar to the microscopic. At longer time scales, the distribution is again approximately exponential; once the molecules have become approximately well-mixed in the enclosing volume the assumption of an exponential reaction rate is quite accurate. This is seen as the second plateau in Figure 3.

In Figure~\ref{fig:rebind-dist-3D} we plot the mean rebinding time as a function of $h$, and show that for $h=\hstarinf$ the mean rebinding times match between the different scales.
 
\begin{figure}[htp]

\centering

\subfigure[Rebinding-time distributions in 3D.]{\includegraphics[width=0.6\linewidth]{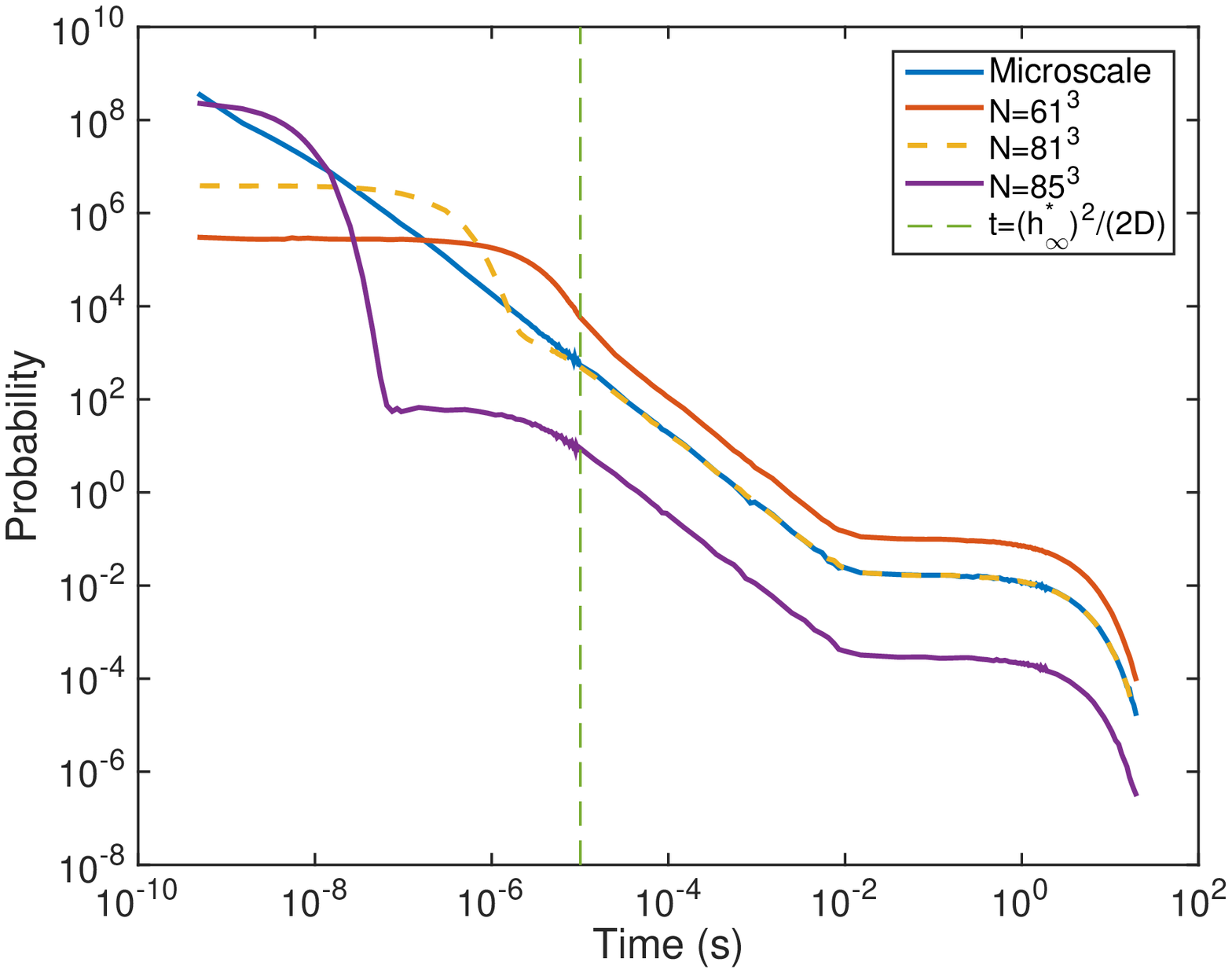}}
\subfigure[Rebinding-time distributions in 2D]{\includegraphics[width=0.6\linewidth]{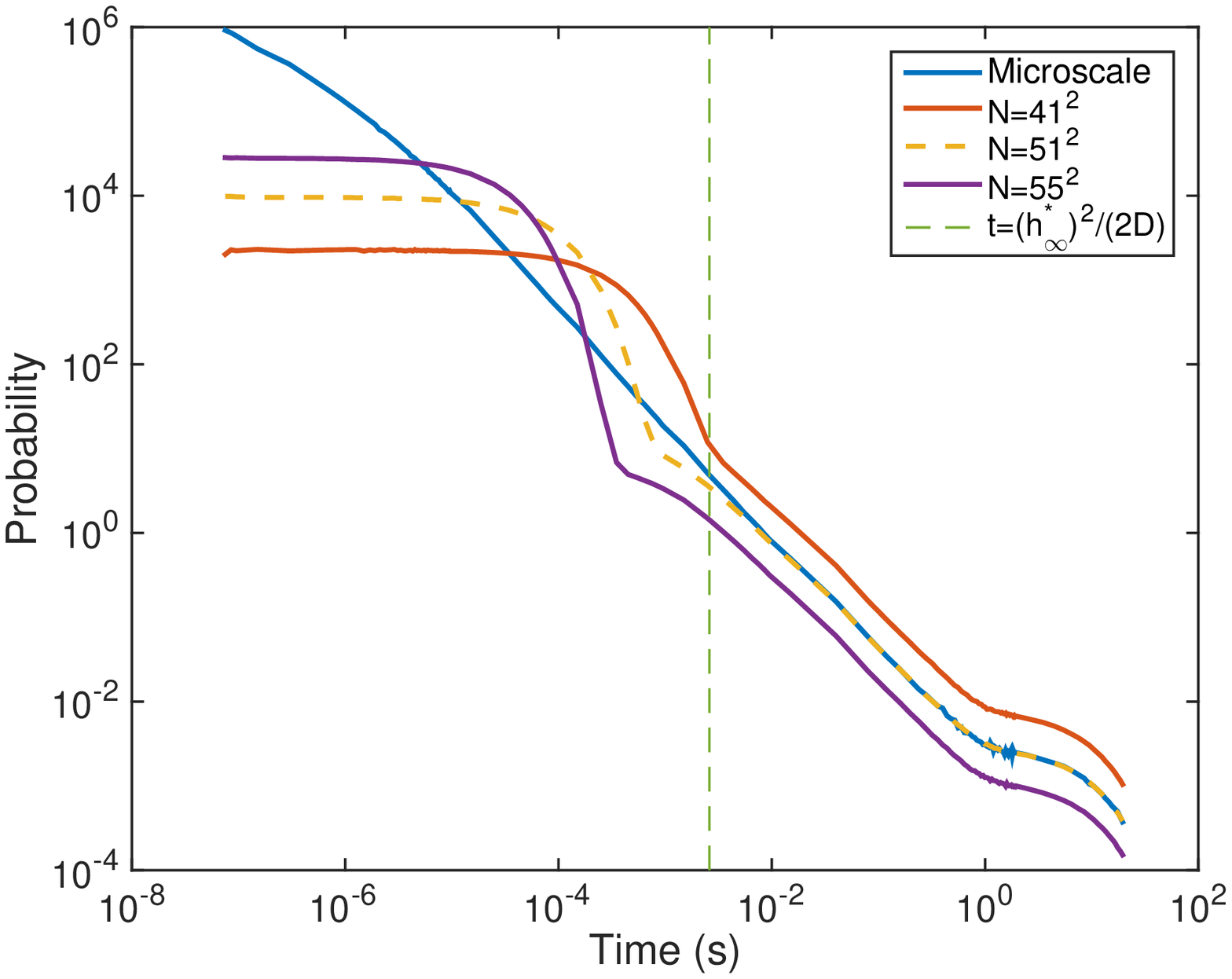}}
\caption{\label{fig:rebind-dist-3D-2D}The width $L=5.145\cdot 10^{-7}$ of the domain has been chosen such that $h\approx \hstarinf$ for $N=81^3$ in (a). In (b), $L=5.2\cdot 10^{-7}$ so that $h\approx\hstarinf$ for $N=51^2$. The mesoscopic rebinding-time distribution matches the microscopic rebinding-time distribution well for $h=\hstarinf$ and $t\gtrsim (\hstarinf)^2/(2D)$. Refining the mesh further, we find that the mean rebinding time decreases, and that the distribution is shifted correspondingly. For coarser meshes, the mean rebinding time increases, and consequently the distribution is shifted in the opposite direction. In (a), the other parameters are given by $\sigma=2\cdot 10^{-9}$, $D = 2\cdot 10^{-12}$, and $\krmi = 10^{-18}$. In (b), the parameters are $\sigma = 2\cdot 10^{-9}$, $D=2\cdot 10^{-14}$, and $\krmi=10^{-12}$.}
\end{figure}

\begin{figure}[htp]
\centering

\subfigure[Mean rebinding time in 3D.]{\includegraphics[width=0.6\linewidth]{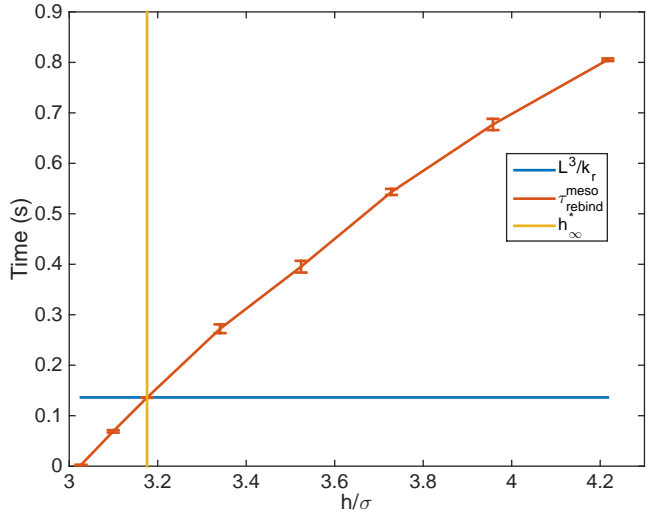}}
\subfigure[Mean rebinding time in 2D.]{\includegraphics[width=0.6\linewidth]{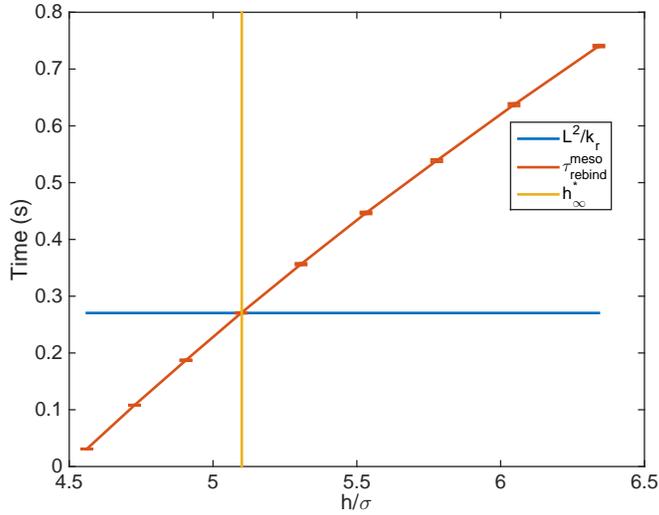}}
\caption{\label{fig:rebind-dist-3D}The mean rebinding times in 3D (a) and 2D (b) as a function of the voxel size $h$. For $h>\hstarinf$ the rebinding time is overestimated, while for $h<\hstarinf$ it is underestimated. We match the mean rebinding time perfectly for $h=\hstarinf$. The parameters are the same as in Figure \ref{fig:rebind-dist-3D-2D}.}
\end{figure}

\subsection{MAPK cascade}
\label{sec:MAPK_example}
An example of when small errors in the transient dynamics of bimolecular equilibration processes can have a large impact on the system's dynamics was given in \cite{TaTNWo10} to illustrate the need of the microscale resolution provided by the GFRD algorithm. The model considered is two steps of the omnipresent mitogen activated phospatase kinase (MAPK) cascade. Here, a transcription factor $MAPK$ is phosphorylated in two steps by a kinase $MAPKK$ and dephosphorylated by a phosphatase $P$:  

\begin{align}
MAPK + MAPKK \overset{k_1}{\underset{k_2}{\rightleftarrows}} MAPK\_MAPKK \\
MAPK\_MAPKK  \xrightarrow{k_3} MAPKK^* + MAPK_p\\
MAPKK^* \xrightarrow{k_7} MAPKK\\
MAPK_p + MAPKK \overset{k_4}{\underset{k_5}{\rightleftarrows}} MAPK_p\_MAPKK \\
MAPK_p\_MAPKK  \xrightarrow{k_6} MAPKK + MAPK_{pp}\\
MAPK_{pp} + P \overset{k_1}{\underset{k_2}{\rightleftarrows}} MAPK_{pp}\_P \xrightarrow{k_3} P^* + MAPK_{p}\\
P^* \xrightarrow{k_7} P\\
MAPK_{p} + P \overset{k_4}{\underset{k_5}{\rightleftarrows}} MAPK_{p}\_P \xrightarrow{k_6} P + MAPK
\end{align}

During the phosphorylation and dephosphorylation steps, the kinase and phosphatase turns into an inactivated form $MAPKK^*$ and $P^*$.  This can model e.g. a conformation change due to conversion of ATP to ADP, resulting in the need to reactivate the enzymes before proceeding with the next reaction. If the timescale for this reactivation step is short and the system very diffusion limited, rapid rebinding of $MAPKK$ to the newly phosphorylated molecules can have a big impact on the overall system dynamics, as illustrated and discussed from a biological perspective in \cite{TaTNWo10}. Numerically, this means that the system is very challenging to simulate with lattice based methods, due to the need for very fine spatial resolution in order to resolve the rebindings on the fast timescale. This was noted by Fange et al. in \cite{FBSE10}, where length scale dependent rates were derived based on the ansatz that the equilibration time should match on the two scales for a spherical discretization. They managed to resolve the microscale dynamics of the model by using these propensities in combination with extending the RDME to allow for reaction events between molecules occupying neighboring voxels as well as molecules occupying the same voxel. Without that extension they were not able to resolve the microscale dynamics; below we show that with reaction rates as defined in \eqref{eq:hhlrate} and \eqref{eq:meso_k_d} we are able to resolve the microscale dynamics without considering a non-local extension of the RDME.

In Figure \ref{fig:mapk_new} we show the results obtained when simulating this model using our local rates for different mesh resolutions. As can be seen, when $h$ is close to $\hstarinf$, we obtain a good approximation of the results of the GFRD algorithm from \cite{TaTNWo10}. In the figure, we show the time $\tau_{\rm res}$ until half-activation (i.e. the time to reach half the steady state level of $MAPK_{pp}$) for varying diffusion constants, making the system range from reaction limited to diffusion limited. As can be seen, for large values of $D$, as expected, the rates proposed here and the rates of Collins and Kimball give similar results, but for the strongly  diffusion limited cases, our rates result in a much better agreement with the microscale model. Notably, for small $D$ we obtain comparable accuracy to using $k_{CK}$ with $h=h_{\infty}^*$ for $h=4h_{\infty}^*$, resulting in simulations that run approximately sixteen times faster, due to the $\mathcal{O}(h^{-2})$ scaling of the computational time.

In Figure \ref{fig:upper_h}a we have computed the upper bound on $h$ obtained by requiring that the relative error in the average rebinding time is bounded by $0.05$. As expected, we can see that for the less diffusion limited cases, when $D/\krmi$ is larger, we can choose the voxel size larger and still obtain accurate results. As $D$ becomes even larger, the well-mixed assumption will be satisfied, and the system can be simulated at a much coarser scale. For smaller values of $D$ the restriction on $h$ is quite severe, and we are required to approach $\hstarinf$ in order to accurately simulate the system. This has an implication for the computational complexity of the simulations; a small $D$ makes the system less stiff, meaning it becomes less computationally challenging, but also forces us to choose a smaller voxel size, while a large $D$ makes the system more stiff, but at the same time allows for a larger voxel size. In Figure \ref{fig:upper_h}b we have computed the maximum error in the average rebinding time, as a function of $D$ but independent of the size $\voxelsize$ of the voxels.

\begin{figure}[htp]
\centering
\includegraphics[width=0.95\linewidth]{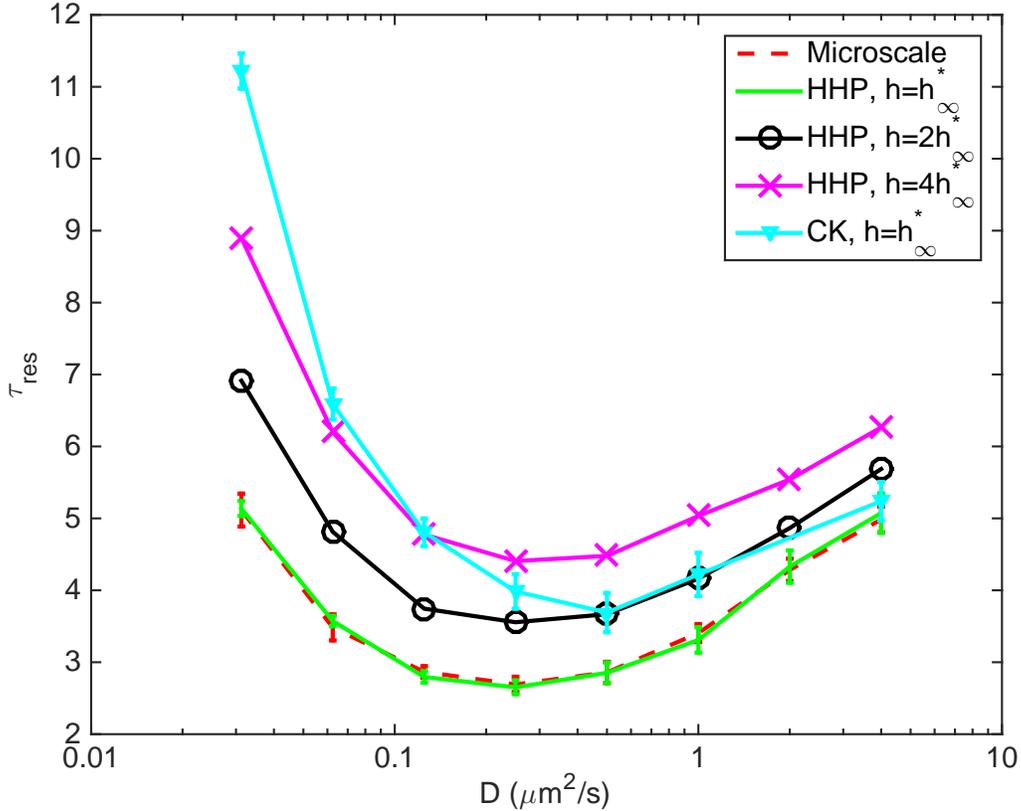}
\caption{\label{fig:mapk_new}We compare the results of mesoscopic simulations using both our proposed multiscale reaction rates (HHP), as well as the classical rates by Collins and Kimball (CK), for different sizes of the mesh. As we can see, for coarser meshes, while capturing the qualitative behavior, we are still not reproducing the microscopic GFRD results accurately. As we refine the mesh and approach $h=\hstarinf$, we approach the microscopic results. The parameters are chosen as in \cite{TaTNWo10}.}
\end{figure}

\begin{figure}[htp]
\centering
\subfigure[Upper bound on $h$ for $\epsilon = 0.05$.]{\includegraphics[width=0.45\linewidth]{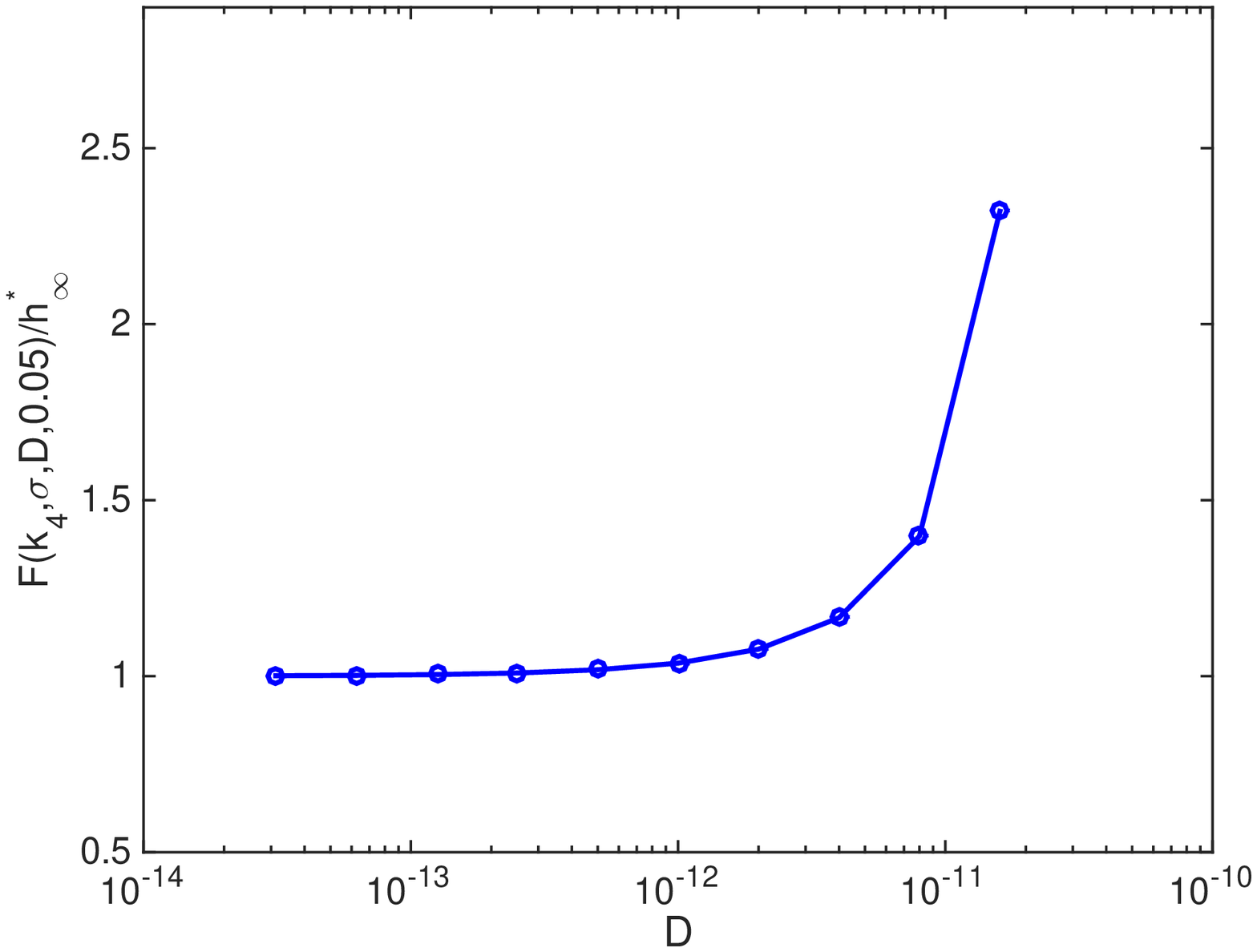}}
\subfigure[$\epsilon_{\rm max}(\diffconst)$]{\includegraphics[width=0.45\linewidth]{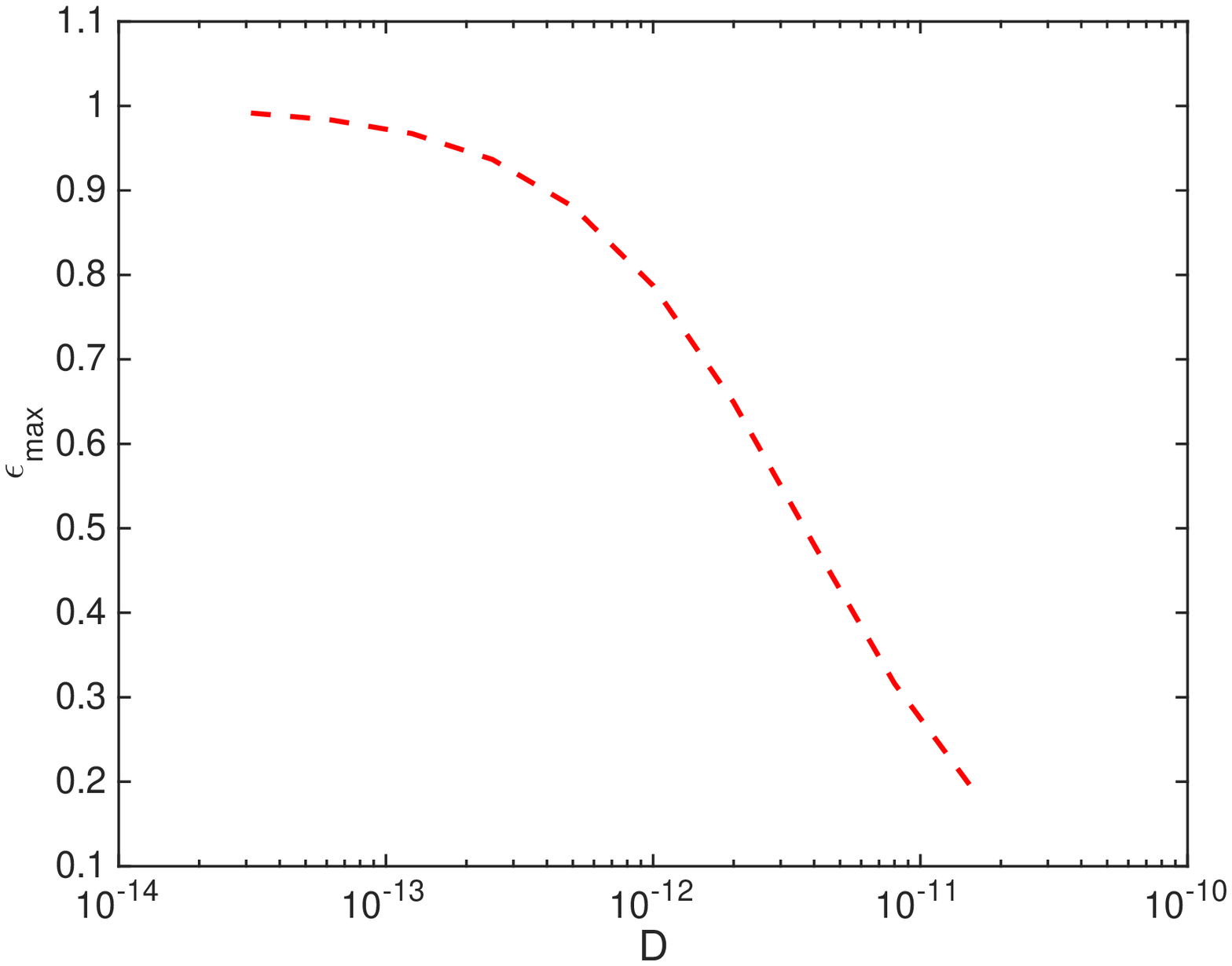}}
\caption{\label{fig:upper_h}As we can see in (a), the restriction on the voxel size $h$ is severe for small $D$. This is seen in Figure \ref{fig:mapk_new} by noting the relatively large error even for smaller $h$. For larger values of $D$, the error is smaller. In accordance with that observation, we see in (a) that the relative error in mean rebinding time decreases with increasing $D$. In (b) we have plotted $\epsilon_{\rm max}$, the maximum error in average rebinding time, as a function of $D$; for larger values of $D$ we note that the maximum relative error in the mean rebinding time is bounded by around $0.19$. This shows that with increasing $D$, as the system gets more and more well-mixed, coarser methods will yield acceptable results. In \cite{TaTNWo10} they show that the microscopic simulations agree with deterministic methods for large enough $D$.}
\end{figure}

\section{Discussion}

As we have seen, by taking a multiscale approach and deriving reaction constants by matching certain statistics of the Smoluchowski model, more accurate simulations can be obtained compared to the classical approach using rates from Collins and Kimball. An important reason why this is possible is that we start out with a given discretization of space, and then derive scale-dependent multiscale propensities (discretization-first), while the CK-rate was not derived with a mesh in mind, and hence the propensity depends only the volume of the voxels and not their shapes. As a consequence, we can better approximate binding times for bimolecular reactions that need high spatiotemporal resolution.  

While it is important to be able to accurately resolve the reaction kinetics of a given diffusion limited system, it may also be important to accurately resolve complex geometries and external and internal boundaries, modeling for example cell membranes. In this case, uniform Cartesian grids have distinct drawbacks compared to unstructured triangular and tetrahedral discretizations in that they require more voxels in order to resolve the boundaries \cite{EnFeHeLo, URDME_BMC} resulting in unnecessarily long computational time. While we expect fundamental limits for a triangular and tetrahedral discretization to be close to the ones obtained herein for Cartesian grids,  whether it will be possible to extend the approaches taken here to obtain sharp estimates and reaction rates also for unstructured meshes has yet to be seen.  Another approach to the problem of making an RDME-type model approximate a microscopic model is taken in \cite{Isaacson}, where a mesoscale model is constructed by discretizing the Doi model \cite{SamConvergent}. This approach seems more directly amenable to be used on general grids, but it has not yet been applied to the Smoluchowski model, and except for the case of irreversible, perfect absorption \cite{SamDoiSDLR}, the relationship between the Doi and Smoluchowski models is not well understood.  

We have illustrated in numerical examples that by using the propensities proposed here it is possible to accurately simulate the MAPK system discussed in \cite{TaTNWo10} on the mesoscale using a purely local RDME implementation, and we have shown theoretically why $\hstarinf$ is the optimal mesh size. The same system has previously been successfully simulated with an RDME-type model by Fange et al. \cite{FBSE10}, using another set of multiscale propensities and by relying on a non-local implementation of the RDME in which molecules occupying adjacent voxels are allowed to react. Relying on neighbor interactions leads to an increased computational cost due to the increased number of updates in each step of the kinetic Monte Carlo algorithm and hence we should expect the propensities derived here to provide a computational speed advantage for comparable accuracy. 

Although more efficient than a non-local implementation, the simulations herein require a uniformly fine mesh and hence expensive simulations for systems that require very high spatial resolution. Unless there are species in the model present in high copy numbers which would cause microscopic simulation to become very time consuming, it is not unlikely that an efficient implementation of e.g. GFRD is more efficient than the purely mesoscopic simulation. With this in mind, for very diffusion limited systems with multiscale properties, a compelling approach is the use of hybrid methods. Such a method, blending the RDME and GFRD algorithms, has previously been proposed by the present authors \cite{HeHeLo}, in which it was demonstrated that an accurate hybrid simulation of the MAPK model \cite{TaTNWo10} can yield accurate results with only a small part of the system simulated on the microscopic scale. However, there are outstanding challenges in making such hybrid methods easy to use by practitioners. Dynamic and adaptive partitioning of the system into microscopic and mesoscopic parts according to accuracy requirements is needed both for computational efficiency and for robustness of the simulator. In this paper we advance the fundamental theoretical understanding of the RDME when viewed as an approximation to the Smoluchowski model, something that is a prerequisite to develop  adaptivity criteria for hybrid methods.  

\section{Acknowledgment}
The authors thank Per L{\"o}tstedt for constructive comments on the manuscript.
This work was funded by NSF award DMS-1001012, NIGMS of the NIH under award R01-GM113241, Institute of Collaborative Biotechnologies award W911NF-09-0001
from the U.S. Army Research Office, NIBIB of the NIH
under award R01-EB014877, and U.S. DOE award DE-
SC0008975. The content of this paper is solely the responsibility of the authors and does not necessarily represent the
official views of these agencies.

\newcommand{\noopsort}[1]{}

\end{document}